\documentclass{amsart}
 \newtheorem{thm}{Theorem}[section]
 \newtheorem{cor}[thm]{Corollary}
 \newtheorem{lem}[thm]{Lemma}
 \newtheorem{prop}[thm]{Proposition}
 \theoremstyle{definition}
 
 \theoremstyle{remark}

 \numberwithin{equation}{section}

\begin{document}

\title[Semisimple Hopf algebras]
{Semisimple Hopf algebras of dimension $2q^3$}

\author[J. Dong]{Jingcheng Dong}

\author[L. Dai]{Li Dai}

\address{College of Engineering, Nanjing Agricultural University, Nanjing
210031, Jiangsu, People's Republic of China} \email[J.
Dong]{dongjc@njau.edu.cn}\email[L. Dai]{daili1980@njau.edu.cn}



\begin{abstract}
Let $q$ be a prime number, $k$ an algebraically closed field of
characteristic $0$, and $H$ a non-trivial semisimple Hopf algebra of
dimension $2q^3$. This paper proves that $H$ can be constructed
either from group algebras and their duals by means of extensions,
or from Radford's biproduct $H\cong R\#kG$, where $kG$ is the group
algebra of $G$ of order $2$, $R$ is a semisimple Yetter-Drinfeld
Hopf algebra in ${}^{kG}_{kG}\mathcal{YD}$ of dimension $q^3$.
\end{abstract}
\maketitle



\section{Introduction}\label{sec1}
In the last twenty years, various classification results were
obtained for finite-dimensional semisimple Hopf algebras over an
algebraically closed field of characteristic $0$. Up to now,
semisimple Hopf algebras of dimension $p,p^2,p^3,pq,pq^2$ and $pqr$
have been completely classified, where $p,q,r$ are distinct prime
numbers. See \cite{Etingof,Etingof2,Masuoka1,Masuoka2,Zhu} for
details.

In the present paper, we shall continue the investigation on the
classification of semisimple Hopf algebras. The main purpose of this
paper is to investigate semisimple Hopf algebras of dimension
$2q^3$, where $q$ is a prime number. Of course, some other
interesting results are also obtained in this paper.

The paper is organized as follows. In Section \ref{sec2}, we recall
the definitions and basic properties of semisolvability, characters,
Radford's biproducts and Drinfeld double, respectively. Some useful
lemmas are also contained in this section. In Section \ref{sec3}, we
study the structure of non-trivial semisimple Hopf algebras of
dimension $pq^3$, where $p,q$ are prime numbers with $p^2<q$. In
Section \ref{sec4}, we study the structure of non-trivial semisimple
Hopf algebras of dimension $2q^3$, where $q$ is a prime number.

Throughout this paper, all modules and comodules are left modules
and left comodules, and moreover they are finite-dimensional over an
algebraically closed field $k$ of characteristic $0$. $\otimes$,
${\rm dim}$ mean $\otimes _k$, ${\rm dim}_k$, respectively. For two
positive integers $m$ and $n$, $gcd(m,n)$ denotes the greatest
common divisor of $m,n$. Our references for the theory of Hopf
algebras are \cite{Montgomery2} or \cite{Sweedler}. The notation for
Hopf algebras is standard. For example, the group of group-like
elements in $H$ is denoted by $G(H)$.

\section{Preliminaries}\label{sec2}

\subsection{Characters}Throughout this subsection, $H$ will be a semisimple Hopf
algebra over $k$.

Let $V$ be an $H$-module. The character of $V$ is the element
$\chi=\chi_V\in H^*$ defined by $\langle\chi,h\rangle={\rm Tr}_V(h)$
for all $h\in H$. The degree of $\chi$ is defined to be the integer
${\rm deg}\chi=\chi(1)={\rm dim}V$. We shall use $X_t$ to denote the
set of all irreducible characters of $H$ of degree $t$. All
irreducible characters of $H$ span a subalgebra $R(H)$ of $H^*$,
which is called the character algebra of $H$. The antipode $S$
induces an anti-algebra involution $*: R(H)\to R(H)$, given by
$\chi\mapsto\chi^*:=S(\chi)$. The character of the trivial
$H$-module is the counit $\varepsilon$.

Let $\chi_U,\chi_V\in R(H)$ be the characters of the $H$-modules $U$
and $V$, respectively. The integer $m(\chi_U,\chi_V)={\rm
dimHom}_H(U,V)$ is defined to be the multiplicity of $U$ in $V$. Let
$\widehat{H}$ denote the set of irreducible characters of $H$. Then
$\widehat{H}$ is a basis of $R(H)$. If $\chi\in R(H)$, we may write
$\chi=\sum_{\alpha\in \widehat{H}}m(\alpha,\chi)\alpha$.

For any group-like element $g$ in $G(H^*)$, $m(g,\chi\chi^{*})>0$ if
and only if $m(g,\chi\chi^{*})= 1$ if and only if $g\chi=\chi$. The
set of such group-like elements forms a subgroup of $G(H^*)$, of
order at most $({\rm deg}\chi)^2$.  See \cite[Theorem 10]{Nichols}.
Denote this subgroup by $G[\chi]$. In particular, we have
$$\chi\chi^*=\sum_{g\in G[\chi]}g+\sum_{\alpha\in \widehat{H},{\rm
deg}\alpha>1}m(\alpha,\chi\chi^*)\alpha.\eqno(2.1)$$

A subalgebra $A$ of $R(H)$ is called a standard subalgebra if $A$ is
spanned by irreducible characters of $H$. Let $X$ be a subset of
$\widehat{H}$. Then $X$ spans a standard subalgebra of $R(H)$ if and
only if the product of characters in $X$ decomposes as a sum of
characters in $X$. There is a bijection between $*$-invariant
standard subalgebras of $R(H)$ and quotient Hopf algebras of $H$.
See \cite[Theorem 6]{Nichols}.

$H$ is said to be of type $(d_1,n_1;\cdots;d_s,n_s)$ as an algebra
if $d_1=1,d_2,\cdots,d_s$ are the dimensions of the simple
$H$-modules and  $n_i$ is the number of the non-isomorphic simple
$H$-modules of dimension $d_i$. That is, as an algebra,  $H$ is
isomorphic to a direct product of full matrix algebras $$H\cong
k^{(n_1)}\times \prod_{i=2}^{s}M_{d_i}(k)^{(n_i)}.$$

If $H^*$ is of type $(d_1,n_1;\cdots;d_s,n_s)$ as an algebra, then
$H$ is said to be of type $(d_1,n_1;\cdots;d_s,n_s)$ as a coalgebra.

\begin{lem}\label{lem1}
Let $\chi$ be an irreducible character of $H$. Then

 (1)\,The order of $G[\chi]$ divides $({\rm deg}\chi)^2$.

 (2)\,The order of $G(H^*)$ divides $n({\rm deg}\chi)^2$, where $n$ is the
 number of non-isomorphic irreducible characters of degree ${\rm deg}\chi$.
\end{lem}
\begin{proof} It follows from Nichols-Zoeller Theorem \cite{Nichols2}. See
also \cite[Lemma 2.2.2]{Natale1}.
\end{proof}

\subsection{Semisolvability}\label{sec2-2}

Let $B$ be a finite-dimensional Hopf algebra over $k$. A Hopf
subalgebra $A\subseteq B$ is called normal if $h_1AS(h_2)\subseteq
A$ and $S(h_1)Ah_2\subseteq A$, for all $h\in B$. If $B$ does not
contain proper normal Hopf subalgebras then it is called simple. The
notion of simplicity is self-dual, that is, $B$ is simple if and
only if $B^*$ is simple.

Let $\pi:H\to B$ be a Hopf algebra map and consider the subspaces of
coinvariants
$$H^{co\pi}=\{h\in H|(id\otimes \pi)\Delta(h)=h\otimes 1\}, \mbox{and\,}$$
$$^{co\pi}\!H=\{h\in H|(\pi\otimes id)\Delta(h)=1\otimes h\}.$$
Then $H^{co\pi}$ (respectively, $^{co\pi}H$) is a left
(respectively, right) coideal subalgebra of $H$. Moreover, we have
$${\rm dim}H ={\rm dim}H^{co\pi}{\rm dim}\pi(H) ={\rm dim}{}^{co\pi}H{\rm dim}\pi(H).$$

The left coideal subalgebra $H^{co\pi}$ is stable under the left
adjoint action of $H$. Moreover $H^{co\pi} ={}^{co\pi}H$ if and only
if $H^{co\pi}$ is a (normal) Hopf subalgebra of $H$. If this is the
case, we shall say that the map $\pi:H\to B$ is normal. See
\cite{Schneider} for more details.

The following lemma comes from \cite[Section 1.3]{Natale4}.
\begin{lem}\label{lem2}
Let $\pi:H\to B$ be a Hopf epimorphism and $A$ a Hopf subalgebra of
$H$ such that $A\subseteq H^{co\pi}$. Then ${\rm dim}A$ divides
${\rm dim}H^{co\pi}$.
\end{lem}

The notions of upper and lower semisolvability for
finite-dimensional Hopf algebras have been introduced in
\cite{Montgomery}, as generalizations of the notion of solvability
for finite groups. By definition, $H$ is called lower semisolvable
if there exists a chain of Hopf subalgebras
$$H_{n+1} = k\subseteq
H_{n}\subseteq\cdots \subseteq H_1 = H$$ such that $H_{i+1}$ is a
normal Hopf subalgebra of $H_i$, for all $i$, and all quotients
$H_{i}/H_{i}H^+_{i+1}$ are trivial. That is, they are isomorphic to
a group algebra or a dual group algebra. Dually, $H$ is called upper
semisolvable if there exists a chain of quotient Hopf algebras
$$H_{(0)} =
H\xrightarrow{\pi_1}H_{(1)}\xrightarrow{\pi_2}\cdots\xrightarrow{\pi_n}H(n)
= k$$ such that $H_{(i-1)}^{co\pi_{i}}$ is a normal Hopf subalgebra
of $H_{(i-1)}$, and all $H_{(i-1)}^{co\pi_i}$ are trivial.

In analogy with the situations for finite groups, it is enough for
many applications to know that a Hopf algebra is semisolvable.

By \cite[Corollary 3.3]{Montgomery}, we have that $H$ is upper
semisolvable if and only if $H^*$ is lower semisolvable. If this is
the case, then $H$ can be obtained from group algebras and their
duals by means of (a finite number of) extensions.

\begin{prop}\label{prop1}
Let $H$ be a semisimple Hopf algebra of dimension $pq^3$, where
$p,q$ are distinct prime numbers. If $H$ is not simple as a Hopf
algebra then it is semisolvable.
\end{prop}
\begin{proof}
By assumption, $H$ has a proper normal Hopf subalgebra $K$.
Moreover, by Nichols-Zoeller Theorem \cite{Nichols2}, ${\rm dim}K$
divides ${\rm dim}H=pq^3$. We shall examine every possible ${\rm
dim}K$.

If ${\rm dim}K=q^2$ or $pq$ then $k\subseteq K\subseteq H$ is a
chain such that $K$ and $H/HK^+$ are both trivial (see
\cite{Etingof,Masuoka1}). Hence, $H$ is lower semisolvable.

If ${\rm dim}K=q^3$ then \cite{Masuoka1} shows that $K$ has a
non-trivial central group-like element $g$. Let $L=k\langle
g\rangle$ be the group algebra of the cyclic group $\langle
g\rangle$ generated by $g$. Then $k\subseteq L\subseteq K\subseteq
H$ is a chain such that $L,K/KL^+$ and $H/HK^+$ are all trivial (see
\cite{Zhu}). Hence, $H$ is lower semisolvable.

If ${\rm dim}K=pq^2$ then  \cite[Lemma 2.2]{dong} and \cite[Theorem
5.4.1]{Natale3} show that $K$ has a proper normal Hopf subalgebra
$L$ of dimension $p,q,pq$ or $q^2$. Then $k\subseteq L\subseteq
K\subseteq H$ is a chain such that $L,K/KL^+$ and $H/HK^+$ are all
trivial. Hence, $H$ is lower semisolvable.

Finally, we consider the case that ${\rm dim}K=p$ or $q$. Let $L$ be
a proper normal Hopf subalgebra of $H/HK^+$ (Notice that $H/HK^+$ is
not simple). Write $\overline{K}=H/HK^+$ and
$\overline{L}=\overline{K}/\overline{K}L^+$.  Then
$H\xrightarrow{\pi_1}\overline{K}\xrightarrow{\pi_2}\overline{L}\xrightarrow{}
k$ is a chain such that every map is normal and $H^{co\pi_1}$,
$(\overline{K})^{co\pi_2}$ are trivial.  Hence, $H$ is upper
semisolvable.
\end{proof}

\subsection{Radford's biproduct}\label{sec2-3}
Let $A$ be a semisimple Hopf algebra and let ${}^A_A\mathcal{YD}$
denote the braided category of Yetter-Drinfeld modules over $A$. Let
$R$ be a semisimple Yetter-Drinfeld Hopf algebra in
${}^A_A\mathcal{YD}$. Denote by $\rho :R\to A\otimes R$, $\rho
(a)=a_{-1} \otimes a_0 $, and $\cdot :A\otimes R\to R$, the coaction
and action of $A$ on $R$, respectively. We shall use the notation
$\Delta (a)=a^1\otimes a^2$ and $S_R $ for the comultiplication and
the antipode of $R$, respectively.

Since $R$ is in particular a module algebra over $A$, we can form
the smash product (see \cite[Definition 4.1.3]{Montgomery2}). This
is an algebra with underlying vector space $R\otimes A$,
multiplication is given by $$(a\otimes g)(b\otimes h)=a(g_1 \cdot
b)\otimes g_2 h, \mbox{\;for all\;}g,h\in A,a,b\in R,$$ and unit
$1=1_R\otimes1_A$.

Since $R$ is also a comodule coalgebra over $A$, we can dually form
the smash coproduct. This is a coalgebra with underlying vector
space $R\otimes A$, comultiplication is given by $$\Delta (a\otimes
g)=a^1\otimes (a^2)_{-1} g_1 \otimes (a^2)_0 \otimes g_2
,\mbox{\;for all\;}h\in A,a\in R, $$ and counit
$\varepsilon_R\otimes\varepsilon_A$.

As observed by D. E. Radford (see \cite[Theorem 1]{Radford}), the
Yetter-Drinfeld condition assures that $R\otimes A$ becomes a Hopf
algebra with these structures. This Hopf algebra is called the
Radford's biproduct of $R$ and $A$. We denote this Hopf algebra by $
R\#A$ and write $a\# g=a\otimes g$ for all $g\in A,a\in R$. Its
antipode is given by
$$S(a\# g)=(1\# S(a_{-1} g))(S_R (a_0 )\# 1),\mbox{\;for
all\;}g\in A,a\in R.$$

A biproduct $R\#A$ as described above is characterized by the
following property(see \cite[Theorem 3]{Radford}): suppose that $H$
is a finite-dimensional Hopf algebra endowed with Hopf algebra maps
$\iota:A\to H$ and $\pi:H\to A$ such that $\pi \iota:A\to A$ is an
isomorphism. Then the subalgebra $R= H^{co\pi}$ has a natural
structure of Yetter-Drinfeld Hopf algebra over $A$ such that the
multiplication map $R\#A\to H$ induces an isomorphism of Hopf
algebras.

Following \cite[Proposition 1.6]{Somm}, $H\cong R\#A$ is a biproduct
if and only if $H^*\cong R^*\#A^*$ is a biproduct.

The following lemma is a special case of \cite[Lemma
4.1.9]{Natale4}.
\begin{lem}\label{lem3}
Let $H$ be a semisimple Hopf algebra of dimension $pq^3$, where
$p,q$ are distinct prime numbers. If $gcd(|G(H)|,|G(H^*)|)=p$, then
$H\cong R\#kG$ is a biproduct, where $kG$ is the group algebra of
group $G$ of order $p$, $R$ is a semisimple Yetter-Drinfeld Hopf
algebra in $^{kG}_{kG}\mathcal{YD}$ of dimension $q^3$.
\end{lem}

\subsection{Drinfeld double}For a finite-dimensional Hopf algebra
$H$, $D(H)=H^{*cop}\bowtie H$ will denote the Drinfeld double of
$H$. $D(H)$ is a Hopf algebra with underlying vector space
$H^{*cop}\otimes H$. More details on $D(H)$ can be found in
\cite[Section 10.3]{Montgomery2}. The following theorem follows
directly from \cite[Proposition 9,10]{Radford2}.

\begin{thm}
Suppose that $H$ is a finite-dimensional Hopf algebra.

(1) The map $G(H^*)\times G(H)\to G(D(H))$, given by
$(\eta,g)\mapsto \eta\bowtie g$, is a group isomorphism.

(2) Every group-like element of $D(H)^*$ is of the form $g\otimes
\eta$, where $g\in G(H)$ and $\eta\in G(H^*)$. Moreover, $g\otimes
\eta\in G(D(H)^*)$ if and only if $\eta\bowtie g$ is in the center
of $D(H)$.
\end{thm}

\begin{cor}\label{cor1}
Suppose that $H$ is a finite-dimensional Hopf algebra such that
$G(D(H)^*)$ is non-trivial. If $gcd(|G(H)|,|G(H^*)|)=1$ then $H$ or
$H^*$ has a non-trivial central group-like element.
\end{cor}
\begin{proof}
Let $1\neq g\otimes \eta\in G(D(H)^*)$. We may assume that $1\neq
g\in G(H)$, since otherwise $\eta\in G(H^*)$ would be a non-trivial
central group-like element, and similarly we may assume that
$\varepsilon\neq\eta\in G(H^*)$. Since $gcd(|G(H)|,|G(H^*)|)=1$, the
order of $g$ and $\eta$ are different. Assume that the order of $g$
is $n$. Then $(g\otimes \eta)^n=g^n\otimes \eta^n=1\otimes
\eta^n\neq1\otimes \varepsilon$ implies that $\eta^n\bowtie 1$ is in
the center of $D(H)$. Hence, $\eta^n$ is a non-trivial central
group-like element in $G(H^*)$. Similarly, we can prove that $G(H)$
also has a non-trivial central group-like element.
\end{proof}

\section{Semisimple Hopf algebras of dimension $pq^3$}\label{sec3}
\begin{lem}\label{lem4}
Let $H$ be a semisimple Hopf algebra of dimension $pq^3$, where
$p<q$ are prime numbers. If $H$ has a Hopf subalgebra $K$ of
dimension $q^3$ then $H$ is lower semisolvable.
\end{lem}
\begin{proof}
Since the index $[H:K]=p$ is the smallest prime number dividing
${\rm dim}H$, the result in \cite{Kobayashi} shows that $K$ is a
normal Hopf algebra of $H$. The lemma then follows from Proposition
\ref{prop1}.
\end{proof}

In the rest of this section, $p,q$ will be distinct prime numbers
such that $p^2<q$, and $H$ will be a non-trivial semisimple Hopf
algebra of dimension $pq^3$.

Recall that a semisimple Hopf algebra $A$ is called of Frobenius
type if the dimensions of the simple $A$-modules divide the
dimension of $A$. Kaplansky conjectured that every
finite-dimensional semisimple Hopf algebra is of Frobenius type
\cite[Appendix 2]{Kaplansky}. It is still an open problem. However,
many examples show that a positive answer to Kaplansky's conjecture
would be very helpful in the classification of semisimple Hopf
algebras.

By \cite[Lemma 2.2]{dong}, $H$ is of Frobenius type and
$|G(H^*)|\neq1$. Therefore, the dimension of a simple $H$-module can
only be $1,p,q$ or $pq$.  It follows that we have an equation
$$pq^3=|G(H^*)|+p^2a+q^2b+p^2q^2c,\eqno(3.1)$$ where $a,b,c$ are the numbers of
non-isomorphic simple $H$-modules of dimension $p,q$ and $pq$,
respectively. By Nichols-Zoeller Theorem \cite{Nichols2}, the order
of $G(H^*)$ divides ${\rm dim}H$.  We shall give some results
concerning the order of $G(H^*)$.

\begin{lem}\label{lem5}
The order of $G(H^*)$ can not be $q$.
\end{lem}
\begin{proof}
Suppose on the contrary that $|G(H^*)|=q$. Let $\chi$ be an
irreducible character of  degree $p$. By Lemma \ref{lem1} (1) and
the fact that $G[\chi]$ is a subgroup of $G(H^*)$, we know that
$G[\chi]=\{\varepsilon\}$ is trivial. It follows that the
decomposition of $\chi\chi^*$  (2.1) gives rise to a contradiction,
since $p^2<q$. Therefore, $a=0$ and equation (3.1) is
$pq^3=q+q^2b+p^2q^2c$, which is impossible.
\end{proof}

\begin{lem}\label{lem6}
If $|G(H^*)|=q^2$ then $a=0$ and $b\neq 0$.
\end{lem}
\begin{proof}
A similar argument as in Lemma \ref{lem5} shows that $a=0$.
Therefore, equation (3.1) is $pq^3=q^2+q^2b+p^2q^2c$. Obviously,
$b\neq0$, otherwise a contradiction will occur.
\end{proof}

\begin{lem}\label{lem7}
If $|G(H^*)|=q^3$ then $H$ is upper semisolvable.
\end{lem}
\begin{proof}
It follows from Lemma \ref{lem4}.
\end{proof}

\begin{lem}\label{lem8}
If $|G(H^*)|=q^2$ and $H$ has a Hopf subalgebra $K$ of dimension
$pq^2$ then $H$ is lower semisolvable.
\end{lem}
\begin{proof}
Considering the map $\pi:H^*\to K^*$ obtained by transposing the
inclusion $K\subseteq H$, we have ${\rm dim}(H^*)^{co\pi}=q$ (see
Section \ref{sec2-2}). By Lemma \ref{lem6}, the dimension of every
irreducible left coideal of $H^*$ is $1,q$ or $pq$. Therefore, by
Lemma \ref{lem2}, as a left coideal of $H^*$, $(H^*)^{co\pi}$
decomposes in the form $(H^*)^{co\pi}=k\langle g\rangle$, where
$\langle g\rangle$ is a subgroup of $G(H^*)$ generated by $g$ which
is of order $q$. Therefore, $(H^*)^{co\pi}$ is a normal Hopf
subalgebra of $H^*$. The lemma then follows from Proposition
\ref{prop1}.
\end{proof}
\begin{lem}\label{lem9}
If $|G(H^*)|=p$ then $H$ is either semisolvable, or isomorphic to a
Radford's biproduct $H\cong R\#kG$, where $kG$ is the group algebra
of $G$ of order $p$, $R$ is a semisimple Yetter-Drinfeld Hopf
algebra in ${}^{kG}_{kG}\mathcal{YD}$ of dimension $q^3$.
\end{lem}
\begin{proof}
Notice that $G(D(H)^*)$ is not trivial by \cite[Lemma 2.2]{dong}. If
$|G(H)|=q^2$ or $q^3$ then $H$ is not simple by Corollary
\ref{cor1}. Hence, $H$ is semisolvable. If $|G(H)|=p,pq$ or $pq^2$
then, by Lemma \ref{lem3}, $H$ is isomorphic to a Radford's
biproduct $H\cong R\#kG$, where $kG$ is the group algebra of $G$ of
order $p$, $R$ is a semisimple Yetter-Drinfeld Hopf algebra in
${}^{kG}_{kG}\mathcal{YD}$ of dimension $q^3$.
\end{proof}

\section{Semisimple Hopf algebras of dimension $2q^3$}\label{sec4}
In this section, $H$ will be a non-trivial semisimple Hopf algebra
of dimension $2q^3$, where $q$ is a prime number. We shall discuss
the structure of $H$. When $q=2$, the structure of $H$ is given in
\cite{Kashina}. When $q=3$, the structure of $H$ is given in
\cite[Chapter 12]{Natale4}. Therefore, in the rest of this section,
we always assume that $q>3$.

Using notations from Section \ref{sec3}, we have
$$2q^3=|G(H^*)|+4a+q^2b+4q^2c.\eqno(4.1)$$
By the results in Section \ref{sec3}, it suffices to consider the
cases that $|G(H^*)|=2q,2q^2$ and $q^2$.

\begin{lem}\label{lem10}
If $|G(H^*)|=2q$ or $2q^2$ then $H$ is either semisolvable, or
isomorphic to a Radford's biproduct $H\cong R\#kG$, where $kG$ is
the group algebra of $G$ of order $2$, $R$ is a semisimple
Yetter-Drinfeld Hopf algebra in ${}^{kG}_{kG}\mathcal{YD}$ of
dimension $q^3$.
\end{lem}
\begin{proof}
By Lemma \ref{lem7} and \ref{lem9}, we may assume that $|G(H)|\neq
2$ and $q^3$.

First, if $gcd(|G(H)|,|G(H^*)|)=2$ then the lemma follows from Lemma
\ref{lem3}.

Second, we assume that $|G(H^*)|=2q$ and $|G(H)|=q^2$. Notice that
$a\neq0$, otherwise equation (4.1) will give rise to a
contradiction. Hence, $G(H^*)\cup X_2$ spans a standard subalgebra
of $R(H)$. It follows that $H$ has a quotient Hopf algebra $K$ of
dimension $2q+4a$. By Nichols-Zoeller Theorem \cite{Nichols2}, ${\rm
dim}K=2q^2$ or $2q^3$. If ${\rm dim}K=2q^2$ then the duality of
Lemma \ref{lem8} proves the lemma. If ${\rm dim}K=2q^3$ then $H$ is
of type $(1,2q;2,a)$ as an algebra. Then \cite[Theorem 6.4]{Bichon}
shows that either $H$ or $H^*$ must contain a non-trivial central
group-like element. The lemma then follows from Proposition
\ref{prop1}.

Finally, we assume that $|G(H^*)|=2q^2$ and $|G(H)|=q^2$. In this
case, the duality of Lemma \ref{lem8} proves the lemma.
\end{proof}

\begin{lem}\label{lem11}
If $|G(H^*)|=q^2$ then $H$ is semisolvable.
\end{lem}
\begin{proof}
By the discussion above, we may assume that $|G(H)|=q^2$.
Considering the map $\pi:H\to (kG(H^*))^*$ obtained by transposing
the inclusion $kG(H^*)\subseteq H^*$, we have ${\rm
dim}H^{co\pi}=2q$. By Lemma \ref{lem6}, the dimension of every
irreducible left coideal of $H$ is $1,q$ or $2q$. Therefore, by
Lemma \ref{lem2}, as a left coideal of $H$, $H^{co\pi}$ decomposes
in the form $H^{co\pi}=k\langle g\rangle\oplus V$, where $k\langle
g\rangle$ is the group algebra of a subgroup $\langle g\rangle$ of
$G(H)$ generated by $g$ which is of order $q$, and $V$ is an
irreducible left coideal of $H$ of dimension $q$. Since $gV$ and
$Vg$ are irreducible left coideals of $H$ isomorphic to $V$ , and
$gV, Vg$ are contained in $H^{co\pi}$, we have $gV=V=Vg$. Then
\cite[Corollary 3.5.2]{Natale4} shows that $k\langle g\rangle$ is a
normal Hopf subalgebra of $k[C]$, where $C$ is the simple
subcoalgebra of $H$ containing $V$, and $k[C]$ is a Hopf subalgebra
of $H$ generated by $C$ as an algebra. Clearly, ${\rm dim}k[C]\geq
q+q^2$. Moreover, by Nichols-Zoeller Theorem \cite{Nichols2}, ${\rm
dim}k[C]=2q^3,q^3$ or $2q^2$. If ${\rm dim}k[C]=2q^3$ then $k[C]=H$
and $k\langle g\rangle$ is a normal Hopf subalgebra of $H$. The
lemma then follows from Proposition \ref{prop1}. If ${\rm
dim}k[C]=q^3$ then the lemma follows from Lemma \ref{lem4}. If ${\rm
dim}k[C]=2q^2$ then the lemma follows from Lemma \ref{lem8}.
\end{proof}

From the discussion in Section \ref{sec3} and \ref{sec4}, we obtain
the main theorem.
\begin{thm}\label{thm2}
Let $H$ will be a non-trivial semisimple Hopf algebra of dimension
$2q^3$, where $q>3$ is a prime number. Then

(1)\, If $gcd(|G(H)|,|G(H^*)|)=2$ then $H$ is isomorphic to a
Radford's biproduct $H\cong R\#kG$, where $kG$ is the group algebra
of $G$ of order $2$, $R$ is a semisimple Yetter-Drinfeld Hopf
algebra in ${}^{kG}_{kG}\mathcal{YD}$ of dimension $q^3$.

(2)\, In all other cases, $H$ is semisolvable.
\end{thm}

As an immediate consequence of Theorem \ref{thm2}, we
 have the following result.
\begin{cor}
If $H$ is simple as a Hopf algebra then $H$ is isomorphic to a
Radford's biproduct $H\cong R\#kG$, where $kG$ is the group algebra
of $G$ of order $2$, $R$ is a semisimple Yetter-Drinfeld Hopf
algebra in ${}^{kG}_{kG}\mathcal{YD}$ of dimension $q^3$.
\end{cor}


\begin{thebibliography}{00}
\bibitem{Bichon}J. Bichon, S. Natale, Hopf algebra deformations of binary polyhedral
groups. Trans. Groups, DOI: 10.1007/s00031-011-9133-x.

\bibitem{dong}J. Dong, Structure of semisimple Hopf algebras of dimension
$p^2q^2$. arXiv:1009.3541v2, to appear in Communications in Algebra.



\bibitem{Etingof}P. Etingof and S. Gelaki, Semisimple Hopf algebras of dimension $pq$ are
trivial. J. Algebra 210(2),  664--669 (1998).

\bibitem{Etingof2}P. Etingof, D. Nikshych and V. Ostrik, Weakly group-theoretical and solvable fusion
categories. Adv. Math. 226 (1), 176--505 (2011).


\bibitem{Kaplansky}I. Kaplansky, Bialgebras. Chicago, University of Chicago
Press 1975.

\bibitem{Kashina}Y. Kashina , Classification of semisimple Hopf algebras of dimension $16$,
J. Algebra 232, 617--663(2000).

\bibitem{Kobayashi}T. Kobayashi and A. Masuoka, A result extended from groups to Hopf
algebras. Tsukuba J. Math. 21(1), 55--58 (1997).


\bibitem{Masuoka1}A. Masuoka, The $p^n$ theorem for semisimple Hopf
algebras. Proc. Amer. Math. Soc. 124, 735--737 (1996).

\bibitem{Masuoka2}A. Masuoka, Self-dual Hopf algebras of dimension $p^3$ obtained by
extension. J. Algebra 178, 791-806 (1995)


\bibitem{Montgomery}S. Montgomery and S. Whiterspoon, Irreducible representations of
crossed products. J. Pure Appl. Algebra 129, 315--326 (1998).


\bibitem{Montgomery2}S. Montgomery, Hopf algebras and their actions on rings. CBMS Reg.
Conf. Ser. Math. 82. Providence. Amer. Math. Soc. 1993.

\bibitem{Natale3}S. Natale, On semisimple Hopf algebras of dimension $pq^r$. Algebras Represent.
Theory 7 (2), 173--188 (2004).

\bibitem{Natale4}S. Natale, Semisolvability of semisimple Hopf algebras of low dimension. Mem. Amer. Math. Soc.
186 (2007).

\bibitem{Natale1}S. Natale, On semisimple Hopf algebras of dimension $pq^2$. J. Algebra
221(2), 242--278 (1999).

\bibitem{Nichols}W. D. Nichols and M. B. Richmond, The Grothendieck group of a Hopf
algebra. J. Pure Appl. Algebra 106, 297--306(1996).

\bibitem{Nichols2}W. D. Nichols and M. B. Zoelle, A Hopf algebra freeness theorem. Amer.
J. Math. 111(2), 381--385 (1989).


\bibitem{Radford}D. Radford, The structure of Hopf algebras with a
projection. J. Algebra 92, 322--347 (1985).

\bibitem{Radford2}D. Radford, Minimal Quasitriangular Hopf algebras. J. Algebra, 157, 285--315 (1993).


\bibitem{Schneider}H.-J. Schneider, Normal basis and
transitivity of crossed products for Hopf algebras. J. Algebra 152,
289--312 (1992).

\bibitem{Somm}Y. Sommerh\"{a}user, Yetter-Drinfel¡¯d Hopf algebras over groups of
prime order, Lectures Notes in Math. 1789, Springer-Verlag (2002).

\bibitem{Sweedler}M. E. Sweedler, Hopf Algebras. New
York, Benjamin 1969.


\bibitem{Zhu}Y. Zhu, Hopf algebras of prime dimension. Internat. Math. Res. Notices
1, 53--59 (1994).
\end{thebibliography}
\end{document}